\pdfoutput=1
\newif\ifpersonal
\personaltrue % comment to remove personal notes

\RequirePackage[l2tabu,orthodox]{nag} %detect whether obsolete packages are used

\documentclass[reqno, 10pt]{amsart}
\usepackage[left=3.5cm,top=3.5cm,right=3.5cm,bottom=3.5cm]{geometry}
\usepackage[dvipsnames]{xcolor}
\usepackage{amsmath,amsthm, amsfonts, amssymb, mathrsfs, bbm}
\usepackage{mathdots}

\usepackage{tikz, tikz-cd} %for drawing, diagrams
\usetikzlibrary{positioning}
\usetikzlibrary{decorations.markings}

\usepackage{verbatim}
\usepackage{microtype,fixltx2e,lmodern} % latex technical issues
\usepackage[utf8]{inputenc} % input encoding
\usepackage[T1]{fontenc} % font encoding
\usepackage{enumitem,comment,braket,xspace,csquotes} % utilities
\usepackage[all,cmtip]{xy} % because the tikzcd options [shift left], [shift right] do not work on arXiv, we switched some diagrams to xymatrix
\usepackage{multirow}
\usepackage{stmaryrd}

\usepackage{tikz, tikz-cd} %for drawing, diagrams
\usetikzlibrary{positioning}
\usetikzlibrary{decorations.markings}

\usepackage[hidelinks]{hyperref}
 \hypersetup{
	colorlinks,
	linkcolor={NavyBlue},
	citecolor={OrangeRed},
	urlcolor={PineGreen}
      }
      \usepackage[capitalize]{cleveref} %now write \cref and it knows if it is a theorem/lemma/etc and it writes it for you.

\theoremstyle{plain}
\newtheorem{proposition}{Proposition}[section]
\newtheorem{lemma}[proposition]{Lemma}
\newtheorem{theorem}[proposition]{Theorem}
\newtheorem*{theorem*}{Theorem}
\newtheorem{corollary}[proposition]{Corollary}
\newtheorem*{corollary*}{Corollary}

\newtheorem{question}[proposition]{Question}
\newtheorem*{question*}{Question}

\theoremstyle{definition}
\newtheorem{definition}[proposition]{Definition}
\newtheorem{remark}[proposition]{Remark}

\newtheorem{example}[proposition]{Example}

\ifpersonal
\newcommand*{\personal}[1]{\textcolor[rgb]{.2,.2,1}{\tiny{(Personal: #1)}}}
\newcommand*{\todo}[1]{\textcolor{red}{(Todo: #1)}}
\else
\newcommand*{\personal}[1]{\ignorespaces}
\newcommand*{\todo}[1]{\ignorespaces}
\fi

\usetikzlibrary{decorations.markings}

\makeatletter
\tikzcdset{
	open/.code     = {\tikzcdset{hook, circled};},
	closed/.code   = {\tikzcdset{hook, slashed};},
	open'/.code    = {\tikzcdset{hook', circled};},
	closed'/.code  = {\tikzcdset{hook', slashed};},
	circled/.code  = {\tikzcdset{markwith = {\draw (0,0) circle (.375ex);}};},
	slashed/.code  = {\tikzcdset{markwith = {\draw[-] (-.0ex,-.4ex) -- (.0ex,.4ex);}};},
	markwith/.code ={
		\pgfutil@ifundefined%
		{tikz@library@decorations.markings@loaded}%
		{\pgfutil@packageerror{tikz-cd}{You need to say %
				\string\usetikzlibrary{decorations.markings} to use arrows with markings}{}}{}%
		\pgfkeysalso{/tikz/postaction = {
				/tikz/decorate,
				/tikz/decoration={markings, mark = at position 0.5 with {#1}}}
		}
	},
}
\makeatother
\newcommand{\rH}{\mathrm H}

\newcommand{\fS}{\mathfrak S}

\newcommand{\fn}{\mathfrak n}

\newcommand{\ft}{\mathfrak t}

\newcommand{\cK}{\mathcal K}

\newcommand{\cO}{\mathcal O}

\newcommand{\cU}{\mathcal U}

\newcommand{\bZ}{\mathbb{Z}}							  				
\newcommand{\bN}{\mathbb{N}}							  				
							  				
\newcommand{\bP}{\mathbb{P}}

\newcommand{\bC}{\mathbb{C}}

\newcommand{\bR}{\mathbb{R}}

\providecommand{\bP}{\mathbb{P}}

\newcommand{\bG}{\mathbb{G}}

\newcommand{\vPhi}{\check{\Phi}}
\newcommand{\vDelta}{\check{\Delta}}
\newcommand{\vmu}{\check{\mu}}
\newcommand{\valpha}{\check{\alpha}}

\newcommand{\Rep}{\operatorname{Rep}}

\makeatletter
\newcommand{\oset}[3][0.2ex]{%
	\mathrel{\mathop{#3}\limits^{
			\vbox to#1{\kern-2\ex@
				\hbox{$\scriptstyle#2$}\vss}}}}
\makeatother

\newcommand{\<}{\langle}
\renewcommand{\>}{\rangle}

\newcommand{\acts}{\curvearrowright}

\usetikzlibrary{decorations.markings} %arrows for open immersions and closed immersions
\tikzset{
  closed/.style = {decoration = {markings, mark = at position 0.5 with { \node[transform shape, xscale = .8, yscale=.4] {/}; } }, postaction = {decorate} },
  open/.style = {decoration = {markings, mark = at position 0.5 with { \node[transform shape, scale = .7] {$\circ$}; } }, postaction = {decorate} }
}
\newcommand{\GL}{\operatorname{GL}}
\newcommand{\SL}{\operatorname{SL}}

\newcommand{\Sp}{\operatorname{Sp}}
\newcommand{\Spin}{\operatorname{Spin}}

\providecommand{\fsl}{\mathfrak{sl}}

\newcommand{\fg}{\mathfrak{g}}

\providecommand{\fb}{\mathfrak{b}}

\newcommand{\Lie}{\operatorname{Lie}}

\DeclareMathOperator{\Hom}{Hom}
\DeclareMathOperator{\End}{End}

\DeclareMathOperator{\Sym}{Sym}

\DeclareMathOperator{\id}{\mathsf{id}}

\DeclareMathOperator{\Ima}{Im}
\DeclareMathOperator{\coker}{coker}
\DeclareMathOperator{\Conv}{Conv}
\DeclareMathOperator{\Stab}{Stab}
\DeclareMathOperator{\height}{ht}

\newcommand{\ind}{\operatorname{ind}}

\newcommand{\res}{\operatorname{res}}

\newcommand{\prim}{\mathrm{prim}}

\makeatletter
\let\save@mathaccent\mathaccent
\newcommand*\if@single[3]{%
	\setbox0\hbox{${\mathaccent"0362{#1}}^H$}%
	\setbox2\hbox{${\mathaccent"0362{\kern0pt#1}}^H$}%
	\ifdim\ht0=\ht2 #3\else #2\fi
}
\newcommand*\rel@kern[1]{\kern#1\dimexpr\macc@kerna}
\newcommand*\widebar[1]{\@ifnextchar^{{\wide@bar{#1}{0}}}{\wide@bar{#1}{1}}}
\newcommand*\wide@bar[2]{\if@single{#1}{\wide@bar@{#1}{#2}{1}}{\wide@bar@{#1}{#2}{2}}}
\newcommand*\wide@bar@[3]{%
	\begingroup
	\def\mathaccent##1##2{%
		\let\mathaccent\save@mathaccent
		\if#32 \let\macc@nucleus\first@char \fi
		\setbox\z@\hbox{$\macc@style{\macc@nucleus}_{}$}%
		\setbox\tw@\hbox{$\macc@style{\macc@nucleus}{}_{}$}%
		\dimen@\wd\tw@
		\advance\dimen@-\wd\z@
		\divide\dimen@ 3
		\@tempdima\wd\tw@
		\advance\@tempdima-\scriptspace
		\divide\@tempdima 10
		\advance\dimen@-\@tempdima
		\ifdim\dimen@>\z@ \dimen@0pt\fi
		\rel@kern{0.6}\kern-\dimen@
		\if#31
		\overline{\rel@kern{-0.6}\kern\dimen@\macc@nucleus\rel@kern{0.4}\kern\dimen@}%
		\advance\dimen@0.4\dimexpr\macc@kerna
		\let\final@kern#2%
		\ifdim\dimen@<\z@ \let\final@kern1\fi
		\if\final@kern1 \kern-\dimen@\fi
		\else
		\overline{\rel@kern{-0.6}\kern\dimen@#1}%
		\fi
	}%
	\macc@depth\@ne
	\let\math@bgroup\@empty \let\math@egroup\macc@set@skewchar
	zen@everymath{\mathgroup\macc@group\relax}%
	\macc@set@skewchar\relax
	\let\mathaccentV\macc@nested@a
	\if#31
	\macc@nested@a\relax111{#1}%
	\else
	\def\gobble@till@marker##1\endmarker{}%
	\futurelet\first@char\gobble@till@marker#1\endmarker
	\ifcat\noexpand\first@char A\else
	\def\first@char{}%
	\fi
	\macc@nested@a\relax111{\first@char}%
	\fi
	\endgroup
}
\makeatother

\newcommand{\eef}{\text{ if }}
\newcounter{dragos}

\newcounter{daniele}

\newcounter{ada}

\newcommand{\Addresses}{{% additional braces for segregating \footnotesize
		\bigskip
		\footnotesize
		A.~Boralevi, Politecnico di Torino, Dipartimento di Scienze Matematiche ``G. L. Lagrange'', 
		Corso Duca degli Abruzzi 24, 10129 Torino, Italy 
		\\ \nopagebreak
		\texttt{ada.boralevi@polito.it}

		\medskip
		\noindent D.~Faenzi, 
                Université Bourgogne Europe, CNRS, IMB UMR 5584,
                F-21000 Dijon, France 
                \\ \nopagebreak
		\texttt{daniele.faenzi@ube.fr}

		\medskip
		\noindent D. Fr\u{a}\c{t}il\u{a}, \textsc{Université de Strasbourg, IRMA, 7 rue René Descartes, 67000 Strasbourg, France}\\\nopagebreak
		\texttt{fratila@math.unistra.fr}
}}

\numberwithin{equation}{section}%equations get numbered according to the section, for example 1.5 or 5.1
\title{Equivariant spaces of matrices of constant corank one}

\author{Ada Boralevi, Daniele Faenzi, Drago\c s Fr\u a\c til\u a}



\begin{document}
	\maketitle
	\begin{abstract}
		We study spaces of matrices coming from irreducible representations of reductive groups over an algebraically closed field of characteristic zero and we completely classify those of constant corank 1.
		In particular, we recover the examples coming from symmetric forms discovered in \cite{BFL2}.
	\end{abstract}

\section{Introduction}

A space of matrices is simply a vector subspace $V$ of the space $\Hom(V_1,V_2)$ of linear maps between two vector spaces $V_1$ and $V_2$. $V$ is said to be of constant rank if all its nonzero elements have the same rank. 
By choosing bases, a space of matrices can be represented by a ``linear matrix'', that is, a matrix whose coefficients are linear forms.
Geometrically it can be represented as a morphism of vector bundles
\begin{equation}\label{eq:intro tildephi}
\tilde\varphi\colon V_1 \otimes \cO_{\bP(V)} \to V_2 \otimes \cO_{\bP(V)}(1).
\end{equation}
The space of matrices has constant rank if and only if the rank of $\tilde\varphi$ is constant, or, equivalently, the cokernel of $\tilde\varphi$ is a vector bundle.
It follows that to a space of matrices of constant rank one can associate several vector bundles on the projective space $\bP(V)$: the kernel, the image and the cokernel of $\tilde\varphi$.

\smallskip

The study of these objects can take different flavors, from
determining an optimal upper bound on the dimension of $V$, to
establishing relations among the possible values of $\dim(V)$,
$\dim(V_i)$, and the rank, to reconstructing the space of matrices
from vector bundles on the projective space, as well as addressing
classification problems. In all its different embodiments, it is a
longstanding problem, tracing back to the work of Kronecker and
Weierstrass. Similarly, the connection to sheaves and vector
bundles---and, consequently, the application of techniques from
projective algebraic geometry---is classical, rooted in the works of
Sylvester \cite{sylvester1986dimension}, Westwick
\cite{westwick-1,westwick-2}, Eisenbud--Harris
\cite{eisenbud1988vector}, Atkinson, Lloyd
\cite{atkinson1983primitive,atkinson-lloyd1981primitive}. Over the
past two decades, spaces of matrices of constant rank have garnered
renewed interest, driven by new ideas and tools emerging from derived
categories  \cite{BFM-instantons}, commutative/homological algebra
\cite{BFL-truncated, simone-rosa},  representation theory \cite{landsberg-manivel,BFL2} or vector bundles on projective spaces \cite{Manivel_Mezzetti,Manivel-Roig-vb}.

\smallskip

Out of the many open questions still standing, in this paper we focus on the construction of examples of spaces that are moreover equivariant under the action of a reductive group. 
Namely, one could take the vector spaces $V$ and $V_i$ to be (irreducible) representations of a given reductive group $G$, and consider a $G$-equivariant inclusion of representations 
\begin{equation}\label{eq:intro phi}
\varphi\colon V \to \Hom(V_1,V_2).
\end{equation}

One of the first instances of this setting appeared in the work \cite{Faenzi07} which led to the classification of all $\SL_2$-equivariant instanton bundles on $\bP^3$.
This approach was later generalized in \cite{BFL2}, where a morphism of type $\varphi$ as above was used to build $\SL_n$-equivariant slope stable vector bundles on the projective space of symmetric forms $\bP(S^d\bC^{n+1})$.

\smallskip

It was observed in \cite{landsberg-manivel} that a space of matrices parameterized by the natural representation of the general linear or symplectic group is always of constant rank (because the action on $\bP(V)$ is transitive and the map $\tilde\varphi$ from \eqref{eq:intro tildephi} is equivariant). 
In loc.cit. the authors exhibit and study many examples of interesting vector bundles on the projective space coming from this observation.

That led them to formulate the following question (see \cite[Question 4]{landsberg-manivel}):  for which irreducible representations $V$, $V_1$, $V_2$ with $V\subset\Hom(V_1,V_2)$ is the resulting space of matrices of constant rank?
It is easy to see that if $V$ is the biggest summand then this is not true.
They suggested that the answer might be positive if $V$ is the smallest summand.
It turns out that this hope is way too optimistic: already for $\SL_2$ and $V$ of dimension bigger than 3 it fails because generically the map $\varphi$ from \eqref{eq:intro phi} is injective, so of maximal rank. 
More generally, one can show that the $\GL(V)$-equivariant space of matrices 
\[\Lambda^kV\to\Hom(\Lambda^rV,\Lambda^{r+k}V)\]
is not of constant rank as soon as $k\ge 2$, despite $\Lambda^kV$
being the smallest summand.
We refer to \cref{P:Lambda not ct rank} for some
details.

\smallskip

It would be very interesting to have a precise conjecture identifying which summands are expected to have constant rank, see \cref{section:concluding remarks} for some more on this.

The examples found in \cite{BFL2} both for $\SL_2$, and more generally for $\SL_n$, focus on the ``second smallest'' summand in the Hom space. Even for this summand to have constant rank, a certain divisibility condition involving the highest weights must be satisfied. The examples we present in this paper also involve this ``second smallest'' summand. 

\smallskip

Our main result gives a classification of all the equivariant spaces of matrices $V\hookrightarrow\Hom(V_1,V_2)$ of constant corank 1 that are associated to irreducible $G$-modules $V$, $V_1$, $V_2$, where $G$ is any reductive group over an algebraically closed field of characteristic zero. 
In detail, we give a necessary and sufficient condition for the (highest weights of the) irreducible modules to yield a linear matrix whose cokernel (see \eqref{eq:intro tildephi}) is a line bundle, and therefore a space of matrices of constant corank 1.

\begin{theorem}\label{T:mainthm-intro}
	Let $V(\nu)$, $V(\mu)$, $V(\lambda)$ be irreducible representations of a reductive group $G$, with $\dim(V(\mu))\le \dim(V(\lambda))$, and suppose that there is an inclusion of representations
	\[ \varphi\colon V(\nu)\to \Hom(V(\mu),V(\lambda)).\]
	Then $\varphi$ is of constant corank 1 if and only if there exists a simple root $\alpha_i$ such that
	\begin{enumerate}
		\item $\lambda = \mu+\nu-\alpha_i$,
		\item $\mu\in \bN\nu$, i.e, $\mu$ is a multiple of $\nu$,
		\item $\nu\in \bN\omega_i$, i.e., $\nu$ is a multiple of $\omega_i$,
	\end{enumerate}
	where $\omega_i$ is the fundamental weight corresponding to $\alpha_i$.
\end{theorem}

\begin{remark}
	The assumption $\dim(V(\mu))\le \dim (V(\lambda))$ is for ``aesthetic reasons''.
	If it is not satisfied, then one can take the transpose and use the isomorphism $V(\lambda)^* \simeq V(-w_0\lambda)$, where $w_0$ is the longest element of the Weyl group, to fall back into the setting of the theorem.
\end{remark}

The proof of Theorem \ref{T:mainthm-intro} occupies the entire Section
\ref{section: main}. The direct implication is fairly straightforward and follows naturally from the weight structure of representations.
The converse, however, is more intricate, and hinges on two key aspects. First, we need to restrict to the closed orbit and inspect the weight structure of representations, which helps us infer the form of the global cokernel. Second, armed with this knowledge we employ additional tools from representation theory to construct the desired space of matrices.

\smallskip

It is worth noting that in the case of $G=\SL_n$ and $i=1$, our construction recovers that of  \cite{BFL2}. Consequently, we also address an implicit question in \cite{BFL2}: whether the specific summand in the decomposition used to construct constant rank matrices is the only possible choice, or if alternative summands are also of constant rank.

\smallskip

\subsection*{Organization of the paper}
We start by recalling in Section
\ref{section: recollection} a series of notions and results regarding representation theory of reductive groups.
Then, in \Cref{section: main} we present our main result and its proof, which is built on a series of subsequent lemmas that we deem interesting on their own right. 
Finally, in \cref{section:concluding remarks} we conclude with some remarks and further questions.

\subsection*{Acknowledgments} A.B. is member of GNSAGA-INdAM (Italy). This study was carried out within the `0-Dimensional Schemes, Tensor Theory, and Applications' project 2022E2Z4AK--funded by European Union--Next Generation EU  within the PRIN 2022 program (D.D. 104 - 02/02/2022 Ministero dell'Universit\`a e della Ricerca).
This research was also supported by the grants DAG-ARTS ANR-24-CE40-4098 of Agence National de la Recherche and by the PNRR Grant CF 44/14.11.2022.
D.F. was partially supported by FanoHK ANR-20-CE40-0023, SupToPhAG/EIPHI
ANR-17-EURE-0002, Bridges
ANR-21-CE40-0017 and "CAPES-COFECUB" programme Ma 1017/24.

\section{Recollection on representations of reductive groups}
\label{section: recollection}

We work over an algebraically closed field of characteristic zero that
we denote by $\bC$. For a finite dimensional vector space $V$ we
denote by $\bP(V)$ the projective space parameterizing lines through
the origin in $V$.

\subsection{}\label{ss:conventions}
Let $G$ be a reductive group and fix $T\subset B\subset G$ a maximal torus and a Borel subgroup.
For simplicity, we only consider the case when the group $G$ is simple and simply connected. 
Actually, this covers the most general case. First, one can always consider the derived subgroup and restrict the representations and no crucial information is lost (irreducibility is preserved and the decomposition of tensor products is the same).
Second, a representation of a semisimple group is also a representation of any of its covering groups.
Lastly, an irreducible representation of a product of simple groups factors through one of the simple groups.

From now on we assume $G$ to be simple and simply connected (e.g., $\SL_n$, $\Sp_{2n}$, $\Spin_n$, $G_2$, etc).
This assumption ensures that the fundamental weights belong to the character lattice.

\subsection{Root system}
We have the lattice of characters $X^*(T):=\Hom(T,\bG_m)$ and the lattice of cocharacters $X_*(T) = \Hom(\bG_m,T)$. We denote the root system by $\Phi$, the coroot system by $\check{\Phi}$, the positive (co)roots (with respect to $B$) by $\Phi^+$ (respectively $\vPhi^+$).
The simple (co)roots are denoted by $\Delta\subset \Phi^+$ respectively $\vDelta\subset\vPhi^+$.

If we write $\fg$ for the Lie algebra of $G$, then the root decomposition of $\fg$ is
\[\fg = \bigoplus_{\alpha\in\Phi}\fg_\alpha\]
where $T$ acts on $\fg_\alpha$ by the character $\alpha$ and every $\fg_\alpha$ is of dimension one.
The Lie algebra of the Borel subgroup is
\[\fb=\Lie(B) = \ft\oplus \oplus_{\alpha\in\Phi^+}\fg_\alpha \]
where $\ft=\Lie(T)$.
We write $N$ for the unipotent radical of $B$. From the above description we have
\[\Lie(N) = \bigoplus_{\alpha\in\Phi^+}\fg_\alpha.\]

The positive roots induce a partial order on the group of characters:
\begin{align}\label{E:partial order characters}
\lambda\preceq\mu \text{ if }\mu-\lambda\in\bN\Delta=\bN\Phi^+.	
\end{align} 

We have a non-degenerate pairing
\[\<-,-\>\colon X^*(T)\otimes X_*(T)\to \bZ=\Hom(\bG_m,\bG_m) \quad \<\lambda,\vmu\>:=\lambda\circ\vmu\]
which allows us to define the dominant characters by
\[ X^*(T)^+:=\{\lambda\in X^*\mid \<\lambda,\valpha\>\ge0,\,\text{ for all positive coroots }\valpha\in\vPhi^+ \}\]

Let us enumerate the simple roots as $\Delta = \{\alpha_i\mid i\in I\}$ and similarly for the simple coroots $\vDelta = \{\valpha_i\mid i\in I\}$.
The fundamental weights are denoted by $\omega_i$ and are characterized by 
\[\<\omega_i, \valpha_j\> = \delta_{i,j} \text{ for all } i,j\in I.\]

\subsubsection{} The Weyl group is defined to be $W:=N_G(T)/T$. It has a structure of Coxeter group determined by the choice of the Borel $B$ (hence of simple roots). The longest element in $W$ is denoted by $w_0$.
It is the unique element $w$ in $W$ with the property $w\Phi^+=\Phi^-$. This implies also that $w_0\vPhi^+=\vPhi^-$ and therefore $w_0$ sends a dominant character to an anti-dominant character: if $\lambda$ is dominant then $-w_0\lambda$ is dominant.

\subsection{Representations}
We need to collect some basic facts about representations of reductive groups.
We will consider representations of $G$ on finite dimensional $\bC$-vector spaces. Sometimes we will refer to them as $G$-modules.
The category of all such representations is denoted by $\Rep(G)$ and
similarly $\Rep(B)$ for representations of $B$. For general results, one can consult, for example, the books \cite{jantzen2003representations,humphreys2012introduction,fulton2013representation}.

Since the category of finite dimensional representations of a semisimple simply connected group is equivalent to the category of finite dimensional representations of its Lie algebra:
\[ \Rep^{\text{fd}}(G)\simeq \Rep^{\text{fd}}(\Lie(G))\]
we allow ourselves to use the point of view that is more convenient to the given context.

The representation theory of a torus is particularly simple: any representation is a direct sum of characters:
\[ V = \bigoplus_{\lambda\in X^*(T)} V_\lambda \]
where, for a $T$-representation $V$, we denoted by $V_\lambda$ the $\lambda$-weight space: the subspace on which $T$ acts through the character $\lambda$. 
Even more concretely, if we note by $\bC(\lambda)$ the vector space $\bC$ with the representation of $T$ given by the character $\lambda$: $t\cdot v = \lambda(t)v$ then the above $V_\lambda$ is isomorphic to $\bC(\lambda)^{\oplus \dim(V_\lambda)}$.
In other words the category $\Rep(T)$ is equivalent to the category of $X^*(T)$-graded vector spaces.

Since we clearly have $\bC(\lambda)\otimes\bC(\mu) = \bC(\lambda+\mu)$, there is an easy Künneth-like decomposition for representations of $T$:
\begin{align}\label{e:kunneth for reps T}
(V\otimes W)_\lambda = \bigoplus_{\mu_1+\mu_2 = \lambda}V_{\mu_1}\otimes W_{\mu_2}.
\end{align}

Any representation of $G$ (or of $B$) can be restricted to a representation of the maximal torus $T$ and, by the previous discussion, becomes a direct sum of characters of $T$.
Therefore any $G$ representation $V$ is graded by the character lattice of $T$
\[ V = \bigoplus_{\mu\in X^*(T)}V_\mu\]
and we call $V_\mu$ the $\mu$-weight space of the representation $V$.
Notice however that the weight spaces $V_\mu$ are moved around by $G$!
The same Künneth type formula \eqref{e:kunneth for reps T} holds for the weight spaces of representations of $G$.

\medskip
\begin{definition}\label{D:primitive}
	A vector $v\in V$ of a $G$-module $V$ is called {\em primitive} if the line it generates is stable under the Borel subgroup $B$.
\end{definition}
In other words, a vector $v$ is primitive if there exists a character $\lambda\colon B\to \bC^*$ such that $b\cdot v = \lambda(b)v$ for all $b\in B$.
In particular, $v$ is a weight vector for $V$ of weight $\lambda$.

Recall that if $V$ is a representation of $G$ then by differentiating it we get a representation of the Lie algebra $\fg$.
We denoted by $\fn=\Lie(N)$ the Lie algebra of the unipotent radical $N$ of $B$.
The following are equivalent characterizations of the notion of primitive
\begin{lemma}\label{L:equivalent primitive vectors}
	Let $V$ be a representation of $G$ and consider a vector $v\in V$.
	Then the following are equivalent
	\begin{enumerate}
		\item $v$ is primitive,
		\item $v$ is fixed by $N$,
		\item the line $\bC v$ is fixed by $\fb$,
		\item $v$ is killed by $\fn$,
		\item $v$ is killed by $\fg_{\alpha_i}$ for all $i\in I$.
	\end{enumerate}
\end{lemma}
\begin{proof}
	The group $B$ is a semidirect product $N\rtimes T$. This is responsible for the first four equivalences.
	The last one follows from the fact that the subspace $\oplus_{i\in I}\fg_{\alpha_i}$ generates the Lie algebra $\fn$.
\end{proof}

By Borel's fixed point theorem, every representation of $B$, hence of $G$, has primitive vectors.
An irreducible representation of $G$ will have a unique (up to scalar) such primitive vector and its weight will be the highest weight of the representation.

For an arbitrary representation $V$ of $G$, we denote by $V_\lambda^\prim$ the subspace of primitive vectors of weight $\lambda$.

\medskip
The irreducible representations of $G$ are indexed by the dominant characters $X^*(T)^+$. We denote the irreducible representation associated to $\lambda\in X^*(T)$ by $V(\lambda)$ -- it is of highest weight $\lambda$ and
the primitive (or highest weight) vector is unique up to multiplication by a scalar. 
We choose a highest weight vector of $V(\lambda)$ that we denote by $v_\lambda$.

For a character $\lambda$ of $T$ (hence also of $B$), we denote by
$\bC(\lambda)$ the corresponding one-dimensional representation of $T$
(or of $B$). The group $T$ (or $B$) acts on it through the character $\lambda$.

We will use the theorem of Borel--Weil on realization of irreducible representations as global sections of line bundles on the flag variety. For convenience we recall the statement below but first we need to introduce the construction of the relevant line bundles.
If $\lambda\in X^*(T)$ is a character of $T$ then it can be extended to a character of $B$ and we can consider the line bundle
\[ \cO(\lambda):=(G\times \bC(-\lambda))/B \to G/B\]
whose global sections are given by
\[ \rH^0(G/B,\cO(\lambda)) = \{f\colon G\to \bC\mid f(gb) = \lambda(b)f(g)\text{ for all }b\in B \text{ and } g\in G\}\]
and are endowed with a natural action of $G$.
\begin{theorem}[{\cite[{Ch. 2}]{jantzen2003representations}}]
	Let $\lambda$ be a dominant weight and let $\cO(\lambda)$  be the associated line bundle on $G/B$.
	Then we have an isomorphism of $G$-modules 
	\[\rH^0(G/B,\cO(\lambda))\simeq V(\lambda)^*.\]
\end{theorem}

More generally, if $M$ is a $B$-module, then one can define the induced $G$-module as follows
\[ \ind_B^G(M):=\{f\colon G\to M\mid f(gb) = b^{-1}\cdot f(g) \text{ for all $b\in B$ and $g\in G$}\}.\]
In particular, notice that $\ind_B^G(\bC(-\lambda)) = \rH^0(G/B,\cO(\lambda))$.

In this context, the Frobenius reciprocity for representations stipulates that we have an adjoint pair of functors 
\[ \res_B^G\colon \Rep(G)\to \Rep(B)\colon \ind_B^G\]
where $\res_B^G$ is simply the restriction of a representation from $G$ to $B$.
See \cite[Proposition 3.4]{jantzen2003representations} for details.
If no confusion can arise, we will denote $\res_B^G(M)$ simply by $M$ as the context will make it clear in what category the object should be viewed.

We will use this reciprocity in the following way:
\begin{lemma}\label{L:basic adjunction}
	Let $\nu\in X^*(T)$ be a dominant character and let $M$ be a representation of $G$.
	Then we have a natural isomorphism of vector spaces
	\[ \Hom_G(V(\nu), M) = \Hom_B(\bC(\nu), M) = M_\nu^{\prim}.\]
\end{lemma}
\begin{proof}
	Write $V(\nu) = \ind_B^G(\bC(\nu)^*)^*$ and use the transpose and Frobenius reciprocity
	\[ 
		\begin{split}
		\Hom_G(\ind_B^G(\bC(\nu)^*)^*, M) &= \Hom_G(M^*,\ind_B^G(\bC(\nu)^*))\\
		&=\Hom_B(M^*,\bC(\nu)^*)\\
		&=\Hom_B(\bC(\nu), M)\\
		&=M_\nu^\prim
		\end{split}
	\]
where the last equality follows from the definition of primitive elements.
\end{proof}

\subsection{Lowest weight representations}
The dual notion of highest weight is called ``lowest weight''.
\begin{definition}\label{D:lowest weight}
A representation $V$ (of $G$ or of $B$) is said to be of lowest weight if it has a weight that is the smallest among all the weights of $V$. 
\end{definition}
For example, the irreducible $G$-representation $V(\lambda)$ is a lowest weight representation with lowest weight $w_0\lambda$.
Equivalently, $V(\lambda)^*$ is a lowest weight representation with lowest weight $-\lambda$.
These notions do not change upon restriction to the Borel subgroup $B$.
For convenience, we recall the following easy lemma:

\begin{lemma}\label{L:lowest weight indec}
Let $V$ be a lowest weight representation of $B$ and suppose that it is generated by a vector of lowest weight.
Then $V$ is an indecomposable representation.
In particular, the quotient of an indecomposable lowest weight $B$-module is still an indecomposable lowest weight $B$-module.
\end{lemma}
\begin{proof}
	Write $V=V_1\oplus V_2$ and let $v$ be the lowest weight vector that generates $V$ as a $B$-representation. Since the weight space corresponding to the lowest weight is of dimension one, precisely one of the $V_1$, $V_2$ must contain $v$. Let us assume it is $V_1$. But since $v$ generates $V$ it follows that $V_1=V$ and hence $V_2=0$. 
\end{proof}

\begin{example}
	The irreducible $G$-representation $V(\lambda)^*$ is of lowest weight $-\lambda$.
	Moreover, as a $B$-representation, it is generated by the lowest weight vector dual of the highest weight vector $v_\lambda\in V(\lambda)$.
	By the above \cref{L:lowest weight indec} we get that $V(\lambda)^*$ is an indecomposable $B$-representation.
	Notice that if $L$ is an arbitrary one-dimensional $B$-module
        and $V$ is a lowest weight $B$-module then $V\otimes L$ is
        still a lowest weight $B$-module. 
	For example $\bC(\nu)\otimes V(\lambda)^*$ is an
        indecomposable lowest weight 
        $B$-module.
\end{example}

The following lemma will play an important role at the end of the proof of our main theorem.
\begin{lemma}\label{L:submodule bigger than -mu}
	Let $V=V(\mu)$ be an irreducible representation of $G$ of highest weight $\mu$ and view it as a $B$-module.
	Then the $B$-submodule $V^*_{\succ-\mu}$ is generated by $\oplus_{i\in I} V^*_{-\mu+\alpha_i}$.
\end{lemma}
\begin{proof}
	Let us denote by $W$ the $B$-submodule $\oplus_{i\in I} V^*_{-\mu+\alpha_i}$.
	We clearly have the inclusion 
	\begin{equation}\label{eq:W sub V^*bigger than -mu}
			W\subseteq V^*_{\succ-\mu}.
	\end{equation}
	We will argue on weight spaces.

	Let us denote by $\xi$ a generator of the weight space $V^*_{-\mu}$.
	As we recalled before, $V^*$ is generated as a $B$-module by $\xi$, hence also as an $N$-module.
	In terms of Lie algebras this is equivalent to
	\[V^* = \cU(\fn)\xi\]
	where $\cU(\fn)$ is the enveloping algebra of the Lie algebra $\fn=\Lie(N)$.
	Hence $V^*$ is a quotient of the algebra $\cU(\fn)$
	and by the PBW theorem, for any $\alpha\in \bN\Phi^*$, we obtain a generating set for the weight space $V^*_{-\mu+\alpha}$ as follows
	\[
	V^*_{-\mu+\alpha} = \<e_{\beta_1}e_{\beta_2}\dots e_{\beta_k}\xi\mid k\ge 1, \beta_l\in\Phi^+\text{ for all }l\>. 
	\]
	In order to show equality in \eqref{eq:W sub V^*bigger than -mu} it is therefore enough to show that for any $k\ge 1$ and any $\beta_1,\dots,\beta_k\in\Phi^+$  the vector
	$e_{\beta_1}e_{\beta_2}\dots e_{\beta_k}\xi$ belongs to $W$.
	
	Since $W$ is a $\cU(\fn)$-submodule we only need to show that the vector $e_\beta\xi$ belongs to $W$ for $\beta\in\Phi^+$. We will show this by induction on the height of $\beta$. Recall that the height of a root $\beta\in\Phi^+$ is the number of simple roots that are needed in its expression.
	
	If $\beta=\alpha_i$ is a simple root then $e_{\alpha_i}\xi\in V_{-\mu+\alpha_i}\subset W$ by definition.
	If $\beta$ is not a simple root then there exist two roots $\beta_1,\beta_2\in\Phi^+$ such that 
	$\beta = \beta_1+\beta_2$. 
	We have $\height(\beta) = \height(\beta_1)+\height(\beta_2)$ hence both $\beta_1$ and $\beta_2$ are of smaller height than $\beta$.
	
	By the structure of reductive Lie algebras we moreover have
	\[ [\fg_{\beta_1},\fg_{\beta_2}] = \fg_{\beta} \]
	which implies that the commutator $[e_{\beta_1},e_{\beta_2}]$ is equal to $ce_\beta$ for a non-zero scalar $c\in\bC^*$.
	We obtain the equality
	\[ ce_\beta\xi = e_{\beta_1}e_{\beta_2}\xi-e_{\beta_2}e_{\beta_1}\xi.\]
	Since both $\beta_1$ and $\beta_2$ are of smaller height than $\beta$, we can use the induction hypothesis to deduce that $e_{\beta_1}\xi$ and $e_{\beta_2}\xi$ belong to $W$ and therefore that $e_\beta\xi$ belongs to $W$.
\end{proof}

\subsection{Representations of $\SL_2$}
We briefly recall some basic facts about the representation theory of $\SL_2$.
It is more convenient to linearize further and work with representations of the Lie algebra $\fsl_2$.

Consider the standard basis $\{e,h,f\}$ of $\fsl_2$ with the usual relations.
\begin{theorem}[{\cite[Theorem7.2]{humphreys2012introduction}}]\label{T:sl_2 reps}
	The irreducible modules of $\fsl_2$ (hence of $\SL_2$) are in bijection with the natural numbers.
	Specifically, a value $n\in \bN$ corresponds to the irreducible representation
	\[ V(n) =\bigoplus_{k=0}^n V(n)_{n-2i}\]
	where $V(n)_{n-2k}$ is a weight space of weight $n-2k$ and dimension one. If we let $v_n\in V(n)_n$ be a highest weight vector then $v_{n-2k}:=f^kv_n$ spans the weight space $V(n)_{n-2k}$ and
	$ev_{n-2k} =k(n-k+1)v_{n-2k+2}$.
\end{theorem}
\begin{corollary}\label{c:strings sl_2 reps}
	If $V$ is a finite dimensional representation of $\fsl_2$ such that for some integers $l$ and $k$ both $l$ and $l+2k$ are weights of $V$ then $l+2m$ is a weight of $V$ for all $m=0,1,\dots,k$.
\end{corollary}
\begin{proof}
	By \cref{T:sl_2 reps} we see that the set of weights of an irreducible representation of $\fsl_2$ is of the form
	\begin{equation}\label{eq:string sl_2}
	\{n-2a\mid a=0,1,\dots,n\} 
	\end{equation} 
	In particular it is symmetric with respect to $0$.
	By the semisimplicity of finite dimensional representations of $\fsl_2$, this will still be the case for $V$.
	
	It is therefore enough to treat the case $l,k\ge 0$.
	The weight $l+2k$ must belong to an irreducible subrepresentation of $V$, and by \eqref{eq:string sl_2}, all the weights of the form
	\[\{l+2k-2m\mid m=0,1,\dots,l+2k\}\]
	must also be weights of this subrepresentation, hence of $V$ and this finishes the proof.	
\end{proof}
\subsection{Root subgroups and more representations}
\subsubsection{} Given a positive root $\alpha\in\Phi^+$ there is a unique subgroup $\SL_2(\alpha)\subset G$ isomorphic to $\SL_2$ containing the image of the coroot $\valpha$ and such that the roots of $\SL_2(\alpha)$ are precisely $\{\pm\alpha\}$. See \cite[Theorem 13.18]{borel2012linear} for details.

In terms of Lie algebras we have 
\[\Lie(\SL_2(\alpha)=\fsl_2(\alpha)=\<f_\alpha,h_\alpha,e_\alpha\>\subset \fg\] 
such that $e_\alpha\in\fg_\alpha$ and $f_\alpha\in \fg_{-\alpha}$.
It follows that $h_\alpha :=[e_\alpha,f_\alpha]\in\Lie(T)$ and this corresponds to the coroot $\valpha$.
For simple roots $\alpha_i$ we will simply denote by $e_i$, $f_i$ the elements $e_{\alpha_i}$, $f_{\alpha_i}$.

\subsubsection{}Since the representation theory of $\SL_2$ is simple enough, it is profitable to restrict a representation of $G$ to the root subgroups $\SL_2(\alpha)$ for various roots $\alpha\in\Phi^+$.
For example, the following is yet another reformulation of the notion of primitive vector (see \cref{D:primitive})
\begin{lemma}\label{L:another equiv primitive}
	Let $V$ be a $G$-module and let $v\in V$. The following are equivalent
	\begin{enumerate}
		\item $v$ is primitive,
		\item $v$ is primitive for the action of $\SL_2(\alpha)$ for all positive roots $\alpha$,
		\item $v$ is primitive for the action of $\SL_2(\alpha_i)$ for all simple roots $\alpha_i$,
		\item $e_i=0$ for all $i\in I$.
	\end{enumerate}
\end{lemma}

Recall also that for a semisimple group, the category of finite dimensional representations is equivalent to the category of modules over its Lie algebra. We will often make use of this equivalence.
Let us note the following consequence of the existence of root subgroups and of the representation theory of $\fsl_2$ (see also \cref{c:strings sl_2 reps})
\begin{proposition}[{\cite[Proposition 3, p.124]{Bour78}}]\label{P:action root subgr}
	Let $V$ be a representation of the group $G$ and denote by $\Omega\subset X^*(T)$ the weights of $V$.
	For $\lambda$ a weight of $V$ and $\alpha$ a positive root of $G$ 
	consider the set of integers
	\[ I:=\{m\in\bZ\mid \lambda+m\alpha\in\Omega\}.\]
	If we put $p=\max(I)$ and $q:=-\min(I)$, then
	\[I=[-q,p] \text{ and } \<\lambda,\valpha\> = q-p.\]
	Said otherwise, we have an inclusion
	\[ \{\lambda-q\alpha,\lambda-(q-1)\alpha,\dots,\lambda+(p-1)\alpha,\lambda+p\alpha\}\subset\Omega \]
	which is moreover maximal.

\end{proposition}
The following corollary will be useful for us
\begin{corollary}\label{C:SL2alpha on lowest weight}
	Let $V$ be an irreducible representation of $G$ of lowest weight $-\lambda$, where $\lambda$ is dominant and
	let $\alpha$ be a positive root. 
	The following are equivalent
	\begin{enumerate}
		\item 	$V_{-\lambda+\alpha}=0$, 
		\item $V_{-\lambda+m\alpha}=0$ for all $m\ge 1$,
		\item $\<\lambda, \valpha\>=0$.
	\end{enumerate}
\end{corollary}
\begin{proof}
	From the previous proposition we have $V_{-\lambda+m\alpha}\neq 0$ precisely for the integers $m$ such that $0\le m\le \<\lambda,\valpha\>]$. The claim follows.
\end{proof}

A special case of the previous corollary can be made more precise

\begin{corollary}\label{C:weight space lambda-alpha_i is of dim at most one}
	If $V$ is an irreducible representation of $G$ of highest weight $\nu$ and $\alpha_i$ is a simple root, then the following are equivalent
	\begin{enumerate}
		\item $V_{\nu-\alpha_i} \neq 0$,
		\item $\dim V_{\nu-\alpha_i} = 1$,
		\item $\<\nu,\valpha_i\>\neq 0$.
	\end{enumerate}
\end{corollary}

For a set of elements $Z\subset X^*(T)$ we denote by $\Conv(Z)\subset X^*(T)$ the convex envelope of $Z$.
It is the intersection of the real convex envelope in $X^*(T)\otimes_\bZ\bR$ with the lattice $X^*(T)$.

Another consequence of the representation theory of $\fsl_2$ is a complete description of all the weights of an irreducible $G$-module:
\begin{corollary}\label{C:weights of irred rep}
	Let $V$ be an irreducible representation of $G$ of highest weight $\lambda$. Denote by $\Omega\subset X^*(T)$ the set of its weights.
	Then we have
	\[ \Omega =  \Conv\{w\lambda\mid w\in W\} \cap \left(\lambda-\bN\Delta\right).\]
\end{corollary}

\section{The main result} \label{section: main}
We are now ready to state and prove our main result: it gives a classification of all the spaces of matrices of constant corank one associated to irreducible $G$-modules.

\begin{theorem}\label{T:main theorem}
Let $\lambda,\mu,\nu\in X^*(T)^+$ be non-zero dominant characters such that there is a morphism of $G$-modules
\[\varphi\colon V(\nu)\hookrightarrow\Hom(V(\mu),V(\lambda)).\]
and assume that $\dim(V(\mu))\le \dim(V(\lambda))$\\
Then $\varphi$ is of constant corank one if and only if there exists $i\in I$ such that the following
conditions are satisfied:
\begin{enumerate}[label=(\alph*)]
	\item \label{thm-c} $\lambda = \nu+\mu-\alpha_i$,
	\item \label{thm-b} $\mu$ is a multiple of $\nu$, i.e., $\mu\in \bN^*\nu$,
        \item \label{thm-a} $\mu$ is a multiple of the fundamental weight $\omega_i$, i.e., $\mu\in\bN^*\omega_i$,

\end{enumerate}
where $\omega_i$ is the fundamental weight corresponding to $\alpha_i$.
\end{theorem}

\begin{remark}
	The assumption $\dim(V(\mu))\le \dim (V(\lambda))$ can be done without any loss of generality; indeed, if it is not satisfied, then one can take the transpose and use the isomorphism $V(\lambda)^* \simeq V(-w_0\lambda)$, where $w_0$ is the longest element of the Weyl group, to fall back into the setting of the theorem.
\end{remark}

\subsection{The direct implication}\label{ss:direct implication}

Here is the synopsis of the proof. We apply first Frobenius reciprocity to $\varphi$ to obtain a map of lowest weight $B$-modules. Then we use the corank one hypothesis, a weight-space argument and the representation theory of $\SL_2$ (\cref{C:SL2alpha on lowest weight}, \cref{C:weights of irred rep}) to deduce points \ref{thm-c} and \ref{thm-a}. 
Lastly, arguing on the first Chern class we show \ref{thm-b}.

\begin{proof}[Proof of the direct implication]
	Suppose we have such a constant corank 1 map $\varphi$. By taking the transpose we can see it as a $G$-equivariant map
\begin{equation}\label{eq:transpose of varphi}
			\varphi\colon V(\nu)\to \Hom(V(\lambda)^*, V(\mu)^*).
\end{equation}
		It is more convenient to work with $\varphi$ under this form in order to avoid using the longest element of the Weyl group $w_0\in W$.
		
	The assumption $\dim(V(\mu))\le \dim(V(\lambda))$ implies that the space of matrices \eqref{eq:transpose of varphi} is of corank 1 precisely when $\dim(\coker(\varphi(u)))=1$ for all $u\in V(\nu)$. 
	
    Using the basic Frobenius reciprocity from \cref{L:basic adjunction} together with the tensor-Hom adjunction, the morphism \eqref{eq:transpose of varphi} is equivalent to a morphism of $B$-modules
	\begin{equation}\label{eq:phi' morph B mod}
		\varphi'\colon \bC(\nu)\otimes V(\lambda)^*\to V(\mu)^*.
	\end{equation}

      \medskip
      
	Both sides are lowest weight indecomposable $B$-modules. The left hand side is a $B$-module of lowest weight equal to $\nu-\lambda$ generated by a single vector, say $v$. 
	Similarly, the right hand side is a $B$-module of lowest weight equal to $-\mu$ generated by a single vector, call it $z$.
	
	The map $\varphi'\colon \bC(\nu)\otimes V(\lambda)^*\to V(\mu)^*$ is uniquely determined by the image of the lowest weight generator $\varphi'(v)$ because the source is an indecomposable $B$-module generated by $v$.
	
	The hypothesis on the corank of $\varphi$ implies that the morphism $\varphi'$ has a dimension one cokernel.
	Since $\varphi'$ is a morphism of $B$-modules it respects the weight spaces and its image is a lowest weight submodule of $V(\mu)^*$ not containing the lowest weight vector.
	In other words, we have shown the equality
	\begin{align}\label{e:image varphi'}
	\Ima(\varphi') = V(\mu)^*_{\succ -\mu}.
	\end{align} 

	Note that by (the dual of) \cref{C:weights of irred rep} the weights of $V(\mu)^*$ are of the form $-\mu+\alpha$ with $\alpha\in\bN\Delta$, and
	from the representation theory of $\SL_2$ (see \cref{C:SL2alpha on lowest weight}), 
	\begin{equation}\label{e:weights of Vmu*}
		-\mu+\alpha_i \text{ is a weight of } V(\mu)^* \text{ if and only if } \<\mu,\valpha_i\>\neq 0. 
	\end{equation}
	Therefore, the sub $B$-module $V(\mu)^*_{\succ -\mu}$ of $V(\mu)^*$ has lowest weights precisely $-\mu+\alpha_i$ for all $i\in I$ such that $\<\mu,\valpha_i\>\neq 0$. (Notice that $-\mu+\alpha_i$, $i\in I$, are pairwise incomparable since $\alpha_i-\alpha_j$ are not positive or negative linear combinations of roots.)

	Putting the last two paragraphs together, we can conclude that there exists a unique $i\in I$ such that $-\mu+\alpha_i$ is a weight of $V(\mu)^*$ and that this is moreover the lowest weight of $\Ima(\varphi')$ and hence of $\bC(\nu)\otimes V(\lambda)^*$.
	In particular, we obtain the equality $\nu-\lambda = -\mu+\alpha_i$, which is precisely point \ref{thm-c} of the theorem we want to prove.
	
	Using \cref{C:SL2alpha on lowest weight} we can moreover deduce that $\<\mu,\alpha_{i'}\>=0$ for all $i'\in I\setminus\{i\}$.
	This is equivalent to $\mu$ being a multiple of the fundamental
        weight $\omega_i$, say $\mu = r\omega_i$, for some non-zero
        $r\in\bN^*$.
	
	The restriction of the $G$-equivariant line bundle $\cO_{\bP(V(\nu))}(n)$ to the $T$-fixed point $[v_\nu]$ is the $T$-module $\bC(-n\nu)$, hence a multiple of the character $\nu$.
	The cokernel of $\varphi$ is such a $G$-equivariant line bundle and its restriction to the $T$-fixed point $[v_\nu]$ is of weight $-\mu$ because it is nothing but the cokernel of $\varphi'$.
	We must therefore have that $-\mu$ is a multiple of $\nu$ which proves \ref{thm-b} and therefore concludes the proof of the direct implication. 
\end{proof}

\subsection{The converse implication}
The proof of this part is a bit more involved.
Here is a summary of the strategy.
First we show, for a dominant character $\nu$ and an arbitrary positive integer $d$, the existence and uniqueness of a non-zero $G$-equivariant morphism
\begin{align*}
\psi\colon V(d\nu)^*\otimes \cO_{\bP(V(\nu))}&\to \cO_{\bP(V(\nu))}(d).
\end{align*}
which is moreover surjective.
Second, under hypothesis \ref{thm-b}, namely $\lambda=d\nu+\nu-\alpha_i$, we show the existence and uniqueness of a non-zero $G$-equivariant morphism
\[\phi\colon V(\lambda)^*\otimes \cO_{\bP(V(\nu)}(-1) \to V(d\nu)\otimes \cO_{\bP(V(\nu))}.\]
We furthermore prove that the composition $\psi\phi$ is zero.

Finally, using the hypothesis \ref{thm-b}, namely $\nu=e\omega_i$ for some $i\in I$, we deduce the exactness of the sequence
\begin{equation*}
	V(\lambda)^*\otimes \cO_{\bP(V(\nu)}(-1)\stackrel{\phi}\to V(d\nu)^*\otimes \cO_{\bP(V(\nu))} \stackrel{\psi}\to \cO_{\bP(V(\nu))}(d) 
\end{equation*}
which allows us to conclude.
\medskip

We start with some lemmas. 

\begin{lemma}\label{L:constr psi}
	Let $\nu$ be a dominant character and let $d\in\bN^*$.
	Then there exists a unique (up to scalar) non-zero $G$-equivariant map of vector bundles
	\[V(d\nu)^*\otimes_\bC \cO_{\bP(V(\nu))} \stackrel{\psi}\to \cO_{\bP(V(\nu))}(d).\]
	Equivalently, there exists a unique (up to scalar) $G$-equivariant inclusion
\begin{equation}\label{eq:unique Vdnu* inside H0}
		 V(d\nu)^*\hookrightarrow \rH^0(\bP(V(\nu)), \cO_{\bP(V(\nu))}(d)).
\end{equation}
\end{lemma}
\begin{proof}
	The data of $\psi$ is equivalent to the data of a
        $G$-invariant section 
	\[ \theta\in \rH^0(\bP(V(\nu)), V(d\nu)\otimes \cO_{\bP(V(\nu))}(d))^G.\]
	Since $\rH^0(\bP,\cO_{\bP(V(\nu))}(d))\simeq \Sym^d(V(\nu)^*)$,
        the section $\theta$ belongs to the $G$-invariant part of $V(d\nu)\otimes\Sym^d(V(\nu)^*)$ which is one dimensional as we will show below.
	
	\medskip
	
	The tensor product $V(\nu)^{\otimes d}$ decomposes as $V(d\nu)\oplus W'$ where $W'$ is a representation of $G$ such that all the highest weights appearing are smaller than $d\nu$. In particular, $V(d\nu)$ appears with multiplicity one in $V(\nu)^{\otimes d}$.
	Since its highest weight vector, namely $v_\nu^{\otimes d}$, is a symmetric tensor we deduce that $V(d\nu)$ is a direct summand in the symmetric power $\Sym^d(V(\nu))$ with multiplicity one.
	We therefore have a canonical decomposition of $G$-representations
	\[ \Sym^d(V(\nu)) = V(d\nu)\oplus W \]
	where $W$ is a representation that does not contain $V(d\nu)$.

	Now it follows that the $G$-invariants in
	\[ V(d\nu)\otimes \Sym^d(V(\nu)^*) = \left(V(d\nu)\otimes V(d\nu)^*\right) \oplus V(d\nu)\otimes W^*.\]
	are reduced to $\End(V(d\nu))^G=\bC\id$.
	This implies the existence and uniqueness (up to scalar) of $\theta$ hence of $\psi$.

\end{proof}

\begin{lemma}\label{L:Vdnu*generates}
	The unique $G$-equivariant linear subspace \eqref{eq:unique Vdnu* inside H0}
	\[ V(d\nu)^*\subset \rH^0(\bP(V(\nu)),\cO_{\bP(V(\nu))}(d))\]
	generates the line bundle $\cO_{\bP(V(\nu))}(d)$.
\end{lemma}
\begin{proof}
We argue by contradiction. Assume there exists a vector $[v] \in \bP(V(\nu))$ such that $\xi(v) = 0$ for all $\xi \in V(d\nu)^*$. By equivariance and continuity, it follows that $\xi$ vanishes on the closure of the orbit $G \cdot [v] \subset \bP(V(\nu))$. Since $V(\nu)$ is irreducible, there is a unique closed $G$-orbit in $\bP(V(\nu))$, namely the orbit of the highest weight vector $[v_\nu]$.  

From this, we conclude the following:  
\begin{align}\label{eq:xiv=0all xi}  
	\xi(v_\nu) = 0 \text{ for all } \xi \in V(d\nu)^*.  
\end{align}  

However, the lowest weight vector $\xi_{-d\nu}$ of $V(d\nu)^* \subset (V(\nu)^*)^{\otimes d}$ is the dual of $v_\nu^{\otimes d} \in V(d\nu) \subset V(\nu)^{\otimes d}$, and thus satisfies  
\[ \xi_{-d\nu}(v_{d\nu}) = \xi_{-\nu}(v_\nu)^d = 1^d \neq 0 \]  
where $\xi_{-\nu} \in V(\nu)^*$ is the dual of $v_\nu \in V(\nu)$.  

This contradicts \eqref{eq:xiv=0all xi}.  
\end{proof}

\begin{corollary}\label{C:surjective onto line bundle}
	Let $\nu$ be a dominant character and $d\in\bN^*$ a positive integer.
	The $G$-equivariant morphism $\psi$ constructed in \cref{L:constr psi} 
\[\psi\colon V(d\nu)^*\otimes\cO_{\bP(V(\nu))}\to \cO_{\bP(V(\nu))}(d)\]
is a surjective morphism of vector bundles.
\end{corollary}
\begin{proof}
	It follows directly from \cref{L:Vdnu*generates}.
\end{proof}

\begin{lemma}\label{L:existence of phi}
	Let $\nu,\lambda\in X^*(T)^+$ and $d\in \bN^*$ be such that 
	\begin{align*}\label{eq:w_0lambda=rnu}
		&\lambda =(d+1)\nu-\alpha_i, \qquad \text{for some $i \in I$ and }\\
		&\<\nu,\valpha_i\>\neq 0.
	\end{align*}
	Then there exists a unique $G$-equivariant morphism
	\[ \phi\colon V(\lambda)^*\otimes\cO_{\bP(V(\nu))}(-1)\to V(d\nu)^*\otimes\cO_{\bP(V(\nu))} \]
\end{lemma}
\begin{proof}
	Notice that such a morphism $\phi$ is uniquely determined by a $G$-equivariant morphism $V(\lambda)^*\to V(d\nu)^*\otimes V(\nu)^*$ which in turn corresponds bijectively (see \cref{L:basic adjunction}) to primitive vectors in $V(d\nu)^*\otimes V(\nu)^*$ of weight $-w_0\lambda$ or, equivalently, to primitive vectors in $V(d\nu)\otimes V(\nu)$ of weight $\lambda$.
	
	We therefore need to show that the representation $V(d\nu)\otimes V(\nu)$ contains a unique (up to scalar) non-zero primitive vector of weight $\lambda$.
	In formula, we want to show the equality
	\[ \dim(V(d\nu)\otimes V(\nu))^\prim_\lambda = 1.\]
	This will follow from $\dim (V(d\nu)\otimes V(\nu))_\lambda = 2$ which is implied (actually equivalent) to the hypothesis $\<\nu,\valpha_i\>\neq 0$.

	Similar to the formula \eqref{eq:weight_space_tensor_prod}, the $\lambda$ weight space of the tensor product is given by a Künneth like formula 
	\begin{equation}\label{eq:weight spaces of Vdnu tensor Vnu}
		 (V(d\nu)\otimes V(\nu))_\lambda = \bigoplus_{\gamma}V(d\nu)_{\lambda-\gamma}\otimes V(\nu)_{\gamma}
	\end{equation}
	where $\gamma$ runs over all characters $X^*(T)$
	(but only finitely many contribute).

	Since the weights of the module $V(\nu)$ are smaller or equal to $\nu$ and similarly for $V(d\nu)$, see \cref{C:weights of irred rep}, we get the following inequalities for the weights $\gamma$ that contribute to \eqref{eq:weight spaces of Vdnu tensor Vnu}
	\begin{align*}
		\gamma\preceq\nu \text{ and }
		\lambda-\gamma\preceq d\nu
	\end{align*}
	which imply, after plugging in the expression for $\lambda$, the inequalities
	\[ \nu-\alpha_i\preceq\gamma\preceq\nu.\]
	We finally deduce that
	\[ \gamma=\nu \text{ or }\gamma = \nu-\alpha_i.\]
	
	Therefore the formula \eqref{eq:weight spaces of Vdnu tensor Vnu} simplifies to
	\begin{equation}\label{eq:VdnutensorVnu dim two}
	(V(d\nu)\otimes V(\nu))_\lambda = V(d\nu)_{d\nu}\otimes V(\nu)_{\nu-\alpha_{i}}\oplus V(d\nu)_{d\nu-\alpha_{i}}\otimes V(\nu)_{\nu}
	\end{equation} 
	which is of dimension two as implied by \cref{C:weight space lambda-alpha_i is of dim at most one} thanks to the hypothesis $\<\nu,\valpha_i\>\neq 0$.
	
	The space of primitive vectors inside $(V(d\nu)\otimes V(\nu))_\lambda$ is, by \cref{L:another equiv primitive}, the joint kernel of
	\begin{equation}\label{eq:ej act on lambda weight space}
		e_j\acts (V(d\nu)\otimes V(\nu))_\lambda \text{ for }j\in I,
	\end{equation}
	or, equivalently, using the expression \eqref{eq:VdnutensorVnu dim two}, the joint kernel of
	\begin{equation}
		e_j\otimes \id+\id\otimes e_j\acts V(d\nu)_{d\nu}\otimes V(\nu)_{\nu-\alpha_{i}}\oplus V(d\nu)_{d\nu-\alpha_{i}}\otimes V(\nu)_{\nu} \text{ for } j \in J.
	\end{equation}

	Since $\nu-\alpha_i+\alpha_j$ is not a weight of $V(\nu)$ for $i\neq j$ and similarly for $d\nu-\alpha_i+\alpha_j$ (see \cref{C:weights of irred rep}) 
	we deduce that the operators $e_j$ from \eqref{eq:ej act on lambda weight space} for $j\neq i$ act by zero.
	Therefore the space of primitive vectors $(V(d\nu)\otimes V(\nu))_\lambda^\prim$ is precisely the kernel of
	\begin{equation}\label{eq:e_i on tens prod dim 1 weight spaces}
		e_i\otimes\id+\id\otimes e_i\acts V(d\nu)_{d\nu}\otimes V(\nu)_{\nu-\alpha_{i}}\oplus V(d\nu)_{d\nu-\alpha_{i}}\otimes V(\nu)_{\nu}
	\end{equation}
	The action of $e_i$ sends $V(\nu)_{\nu-\alpha_i}$ bijectively to $V(\nu)_\nu$ and is zero on $V(\nu)_\nu$; similarly for $d\nu$.
	Hence the image of the operator \eqref{eq:e_i on tens prod dim 1 weight spaces} must be $(V(d\nu)\otimes V(\nu))_{(d+1)\nu}$ which is of dimension one. 
	We can conclude that the kernel of \eqref{eq:e_i on tens prod dim 1 weight spaces} is of dimension one.
	In other words, the space $(V(d\nu)\otimes V(\nu))_\lambda^\prim$ contains a unique (up to scalar) non-zero primitive weight vector.
			
	In conclusion there is a unique line of highest weight $\lambda$ in $V(d\nu)\otimes V(\nu)$ and this settles the existence and uniqueness from the statement.
\end{proof}

\begin{lemma}\label{L:no map Vlambda to Symr}
	Let $\nu$ be a dominant character, $r\ge 1$  a positive integer and $i\in I$ such that $\lambda:= r\nu-\alpha_i$ is also a dominant character.
	Then 
	\[ \Hom_{G}(V(\lambda),\Sym^r(V(\nu))) = 0 = \Hom_G(\Sym^r(V(\nu)),V(\lambda)).\]
\end{lemma}
\begin{proof}
	The two statements are equivalent since the category of representations of $G$ is semisimple.
	
	We need to show that $\Sym^r(V(\nu))$ does not contain an irreducible summand of the form $V(r\nu-\alpha_i)$.
	From \cref{L:basic adjunction} this is equivalent to showing that the space of primitive elements $\Sym^r(V(\nu))^\prim_\lambda$ of weight $\lambda$ is zero.
	
	We will write $V$ for $V(\nu)$ to simplify the notation.
	
	Since $\Sym^r(V)$ are the $\fS_r$-invariants in the representation $V^{\otimes r}$, let us first look at the $\lambda$ weight space of $V^{\otimes r}$.
	By the Künneth-type formula \eqref{e:kunneth for reps T}, we have 
	\begin{equation}\label{eq:weight_space_tensor_prod}
	(V^{\otimes r})_{\lambda} = \bigoplus_{\mu_1,\dots,\mu_r} V_{\mu_1}\otimes\dots\otimes V_{\mu_r} 
	\end{equation} 
	where the sum runs over all characters $\mu_1,\dots,\mu_r$ that sum to $\lambda$.
	Notice that by \cref{C:weights of irred rep} we have $V_{\mu_1}=0$ for a character $\mu_1$ that is not smaller than $\nu$.
	
	Since $\lambda = r\nu-\alpha_i$ the above formula simplifies to
	\[ (V^{\otimes r})_{\lambda} = \bigoplus_{l=1}^rV_\nu\otimes\dots\otimes 
	\underbrace{V_{\nu-\alpha_i}}_{\text{position }l}\otimes\dots\otimes V_\nu.\]
	
	Taking $\fS_r$ invariants we obtain
	\[ (\Sym^rV)_\lambda = \left(\bigoplus_{l=1}^rV_\nu\otimes\dots\otimes 
	\underbrace{V_{\nu-\alpha_i}}_{\text{position }l}\otimes\dots\otimes V_\nu\right)^{\fS_r} \]
	where $\fS_r$ acts by permuting the tensor factors.
	
	If $V_{\nu-\alpha_i}=0$, i.e., if $\<\nu,\valpha_i\> = 0$ (see \cref{C:weight space lambda-alpha_i is of dim at most one}), then there is nothing to show.
	
	Otherwise, $\dim V_{\nu-\alpha_i}=1$ by \cref{C:weight space lambda-alpha_i is of dim at most one}. 

		The weight spaces $V_\nu$ and $V_{\nu-\alpha_i}$ are one dimensional. Call $v_\nu$ and $v_{\nu-\alpha_i}$ some generators. 
		We obtain
	\[ \Sym^r(V)_\lambda = \bC \left\{\sum_{l=1}^r v_\nu\otimes\dots\otimes\underbrace{v_{\nu-\alpha_i}}_{\text{position }l}\otimes\dots\otimes v_\nu\right\}\]
	which proves that this weight space is one dimensional.

	We are now ready to conclude: the representation $V(r\nu)$ is already contained in $\Sym^r(V)$ and hence contains the primitive vector $v_\nu^{\otimes r}$ of weight $r\nu$.
	Applying the operator $f_i\in\fsl_2(\alpha_i)\subset \Lie(G)$ to it gives a non-zero element of $\Sym^r(V)_\lambda$ which is one dimensional. 
	Hence $\Sym^r(V)_\lambda$ does not contain primitive vectors which is what we wanted to prove.
\end{proof}

\begin{lemma}\label{L:psiphi is zero}
	Let $\nu$ be a dominant weight, $\alpha_i$ a simple root and $d\ge 1$ a positive integer.
	Assume that $\<\nu,\valpha_i\>\neq 0$ and put $\lambda = (d+1)\nu-\alpha_i$.
	
	Then the composition of $\psi$ constructed in \cref{L:constr psi} and of $\phi$ constructed in \cref{L:existence of phi} 
	\begin{align}\label{eq:phipsi sequence}
		V(\lambda)^*\otimes\cO_{\bP(V(\nu))}(-1)\stackrel\phi\to V(d\nu)^*\otimes\cO_{\bP(V(\nu))}\stackrel{\psi}{\to}\cO_{\bP(V(\nu))}(d)
	\end{align}
	is the zero morphism.
\end{lemma}
\begin{proof}
	We use the following general fact: a morphism between coherent sheaves generated by global sections is zero if and only if it is so on global sections.
	
	We tensor by $\cO(1)$ the sequence \eqref{eq:phipsi sequence} and notice that all the resulting vector bundles are generated by global sections. 
	Using the above remark we need to show that the following composition is zero
	\[ V(\lambda)^*\to V(d\nu)^*\otimes V(\nu)^*\to \Sym^{d+1}(V(\nu)^*).\]
	This follows at once from \cref{L:no map Vlambda to Symr} by taking duals.
\end{proof}

We are now ready to finish the proof of our main \cref{T:main theorem}.
\begin{proof}[Proof of the converse implication in \cref{T:main theorem}]
	Suppose $\lambda,\mu,\nu$ satisfy the conditions \ref{thm-c}, \ref{thm-a} 
        and \ref{thm-b}.
	So there exists a positive integer $d\ge 1$ such that $\lambda = (d+1)\nu-\alpha_{i}$ and $\mu = d\nu$.
	
	A $G$-equivariant morphism 
	\[V(\nu)\to \Hom(V(\mu),V(\lambda))\] 
	as in the \cref{T:main theorem} is equivalent to a $G$-equivariant morphism of vector bundles
	\[ \phi\colon V(\lambda)^*\otimes\cO_{\bP(V(\nu))}\to V(d\nu)^*\otimes \cO_{\bP(V(\nu))}(1) \]
	whose existence and uniqueness is guaranteed by 
	\cref{L:existence of phi}.
	
	Applying \cref{L:psiphi is zero} we obtain that $\phi$ has generic corank at least one.
	The action of $G$ on $\bP(V(\nu))$ has a unique closed orbit and the corank of $\phi$ is maximal on this closed orbit.
	In order to show that $\phi$ has constant corank one it is therefore enough to see that on the closed orbit the corank is one.
	
	The restriction of $\phi$ to the closed orbit is a morphism 
	\[ \phi'\colon V(\lambda)^*\otimes \bC(\nu)\to V(\mu)^* \]
	of $P(\nu)$-representations, where $P(\nu):=\Stab_G([v_\nu])$ is the parabolic subgroup associated to the character $\nu$.
	In particular $\phi'$ is a morphism of lowest weight $B$-modules (see \cref{D:lowest weight}).
	
	Arguing as we did for the morphism \eqref{eq:phi' morph B mod} and using the hypothesis $\mu\in\bN^*\omega_i$ we have a decomposition
	\begin{equation}\label{eq:Vmu* decomposed}
		V(\mu)^* = V(\mu)^*_{-\mu}\oplus V(\mu)^*_{\succeq-\mu+\alpha_i}
	\end{equation}
		On the one hand, the image of the morphism $\phi'$ is a lowest weight $B$-submodule of $V(\mu)^*$ of lowest weight $-\lambda+\nu = -\mu+\alpha_i$, so we have an inclusion
	\[ \Ima(\phi')\subseteq V(\mu)^*_{\succeq -\mu+\alpha_i}.\]
	
	On the other hand, by \cref{C:weight space lambda-alpha_i is of dim at most one}, we have $\dim V(\mu)^*_{-\mu+\alpha_i}=1$ and since $V(\mu)^*_{-\mu+\alpha_i}\subset W$, \cref{L:submodule bigger than -mu} implies the equality 
	\begin{equation}\label{eq:image phi'}
		\Ima(\phi') = V(\mu)^*_{\succeq-\mu+\alpha_i}.
	\end{equation}
	Given that $\dim V(\mu)^*_{-\mu}=1$, the two formulae \eqref{eq:Vmu* decomposed}, \eqref{eq:image phi'} imply that $\phi'$ has corank one, thus concluding the proof.
\end{proof}

\section{Remarks and questions}\label{section:concluding remarks}
\subsection{}
Let us recall the following question posed by Landsberg and Manivel in \cite[Question 4]{landsberg-manivel} (see also our discussion in the %\cref{ss:intro-landsb-maniv} from the 
introduction):

\begin{question*}	
For which irreducible representations $V$, $V_1$ and $V_2$ of a reductive group such that $V\subset \Hom(V_1,V_2)$ is the space of matrices of constant rank?
\end{question*}

They suggest that the smallest summand $V\subset \Hom(V_1,V_2)$ could be a possible answer, but---as we mentioned in the introduction---this fails already for $\SL_2$ and $\dim(V)\ge 3$.
In the next subsection we give another such non-example.

\subsection{Wedge powers}
For an $n$-dimensional vector space $V$ and for integers $r,k\ge 1$ such that $r+k<n$ consider the following $\GL(V)$-equivariant space of matrices
\[\theta_{k,r}\colon\Lambda^kV\to \Hom(\Lambda^rV, \Lambda^{r+k}V).\]
For $k=1$ this space is clearly of constant rank but how about for $k\ge 2$?

The vector spaces $\Lambda^kV$, $\Lambda^rV$, $\Lambda^{r+k}V$ are irreducible representations of $\GL(V)$ (actually they are fundamental representations) and $\Lambda^kV$ is minimal in $\Hom(\Lambda^rV,\Lambda^{r+k}V)$ because its highest weight, a fundamental weight, is minimal in the poset of dominant weights.
However, we have
\begin{proposition}\label{P:Lambda not ct rank}
	The space $\theta_{k,r}$ is not of constant rank if $k\ge 2$.
\end{proposition}
\begin{proof}
Pick a basis $e_1,\dots,e_n$ of $V$ and consider the following two elements of $\Hom(\Lambda^rV,\Lambda^{r+k}V)$:
\begin{align*}
	E&:=\theta_{k,r}(e_{1,\dots,k})\\
	E'&:=\theta_{k,r}(e_{n-k+1,\dots,n})
\end{align*}
The idea is to show that $\dim\ker(E)$ is strictly bigger than $\dim\ker(E+E')$. This implies that $\theta_{k,r}$ is not of constant rank.

There are two cases to consider: 
\begin{enumerate}
	\item case 1: $2k\le n$,
	\item case 2: $2k>n$.
\end{enumerate}

Both situations use the same strategy for computation that we sketch below (for the first case).
Decompose $V$ into a direct sum as follows
\[ V=\<e_1,\dots,e_k\>\oplus \<e_{k+1},\dots,e_{n-k}\>\oplus\<e_{n-k+1},\dots,e_n\>=:V_1\oplus V_2\oplus V_3.\]
For natural numbers $a,b,c$ put
\[W_{a,b,c}:=\Lambda^aV_1\otimes\Lambda^bV_2\otimes\Lambda^c V_3.\]
Then we have the decomposition
\[ \Lambda^r V = \bigoplus_{a+b+c = r} W_{a,b,c}.\]

This decomposition is compatible with the operators $E$ and $E'$ in the following sense
\begin{align}
	E(W_{a,b,c}) &= \left\{\begin{array}{ll}
		0 & \eef a\neq 0,\\
		W_{k,b,c} & \eef a=0,
	\end{array}\right.\\
	E'(W_{a,b,c})& = \left\{
	\begin{array}{ll}
		0 & \eef c\neq 0,\\
		W_{a,b,k} & \eef c=0.
	\end{array}
	\right.
\end{align}
We deduce an inclusion
\[\bigoplus_{\substack{a+b+c=r\\ a\neq0\neq c}} W_{a,b,c}\subseteq\ker(E+E').\]
The rest of the kernel is isomorphic to $W_{0,r-k,k}$ sitting (sign) diagonally in $W_{k,r-k,0}\oplus W_{0,r-k,k}$.
In any case, we get $\dim\ker(E)>\dim\ker(E+E')$.
\end{proof}

\subsection{PRV components?}
It would be highly desirable to identify some conjectural candidates $V$, $V_1$, $V_2$ for which the space $V\hookrightarrow\Hom(V_1,V_2)$ is of constant rank. Although the suggestion by Manivel and Landsberg to consider the smallest summand does not work in general, it may still be valid in certain special cases, depending on the highest weights of the irreducible representations involved.

It is important to note that the irreducible summands in $\Hom(V_1,V_2)$ appear, in general, with multiplicity. If the multiplicity is not one, then there is ambiguity in selecting an irreducible summand and the answer to the question will depend on this choice.
However, in some special cases, it is possible to single out some canonical irreducible subrepresentations, called the PRV components (see \cite{Kumar-PRV}).
Even though these components may also appear with higher multiplicity, there is a way to select a canonical summand.

Switching back to highest weight notation, given two irreducible modules $V(\lambda)$ and $V(\mu)$, the PRV conjecture (proved in \cite{Kumar-PRV}) claims that for any $w$ in the Weyl group $W$, the irreducible representation $V(\overline{\mu-w\lambda})$ of extremal weight $\mu-w\lambda$ appears with non-zero multiplicity in the representation $\Hom(V(\lambda),V(\mu))$.
There is a strengthened version (see loc.cit.) suggested by Kostant claiming that inside a natural submodule $M_{\lambda,\mu}^w\subset \Hom(V(\lambda),V(\mu))$ the multiplicity is actually one!
We denote the inclusion of this well defined summand by
\begin{equation}	\label{eq:PRV comp}
\varphi_{\lambda,\mu}^w\colon V(\overline{\mu-w\lambda})\hookrightarrow\Hom(V(\lambda),V(\mu)).
\end{equation} 
The smallest summand is obtained by taking $w=1$ and the biggest summand by taking $w=w_0$ (actually they have multiplicity one in the full Hom representation).

In line with Landsberg--Manivel's suggestion, we propose the concrete question
\begin{question}
	For which dominant weights $\lambda,\mu$ and Weyl group element $w\in W$ is the canonical PRV component \eqref{eq:PRV comp}
	of constant rank?
\end{question}

\subsection{The case of $\SL_2$}
We have attempted to fully understand the situation in the case of $\SL_2$, but a definite answer eluded us. For clarity, let us give some details.

For a positive integer $n\in\bN$ we denote by $V(n)$ the irreducible representation of $\SL_2$ of highest weight $n$. 
The Clebsch--Gordan formula for $n\ge m$ reads:
\begin{equation}\label{eq:clebsch-gordan}
	\Hom(V(m),V(n)) \simeq \bigoplus_{k=0}^m V(n-m+2k).
\end{equation} 
The smallest summand $V(n-m)$ is of maximal rank as can be seen by exhibiting its highest weight vector inside the Hom space.

Actually, it is not hard to see that the highest weight vector of $V(n-m+2k)$ in
\eqref{eq:clebsch-gordan} is of the form:

  \[ \begin{pmatrix}
	0&\dots&&&&&&&&0\\
	\vdots&&&&&&&&&\vdots\\
	&&&&&&&&&0\\
	&&&&&&&&&(-1)^{m-k}\\
	&&&&&&&&\iddots&0\\
	&&&&&&&\qquad\quad\iddots&&0\\
	&&&&&&&(-1)^i{m-k\choose i}&0&\vdots\\
	&&&&&&\iddots&&\vdots&\vdots\\
	0&\cdots&&&0&1&0&\cdots&\cdots&0
\end{pmatrix} \in M_{m+1,n+1}(\bC)\]

Therefore the corank of such a
matrix is $k$. 
The matrix corresponding to the lowest weight vector will be the rotation of the above matrix by $180^\circ$.
In order for their sum to still have corank $k$, the following inequality needs to hold
\[ 2m-n+1\ge 3k.\]
We can therefore deduce
\begin{corollary}
	Let $m\le n$ be natural numbers and $k$ be an integer between $0$ and $m$.
	If $3k>2m-n+1$, then the summand
	\[ V(n-m+2k)\hookrightarrow\Hom(V(m),V(n)) \text{ is not of constant rank}.\]
\end{corollary}

In \cite{BFL2} it was shown that for $n\ge m$ and $n-m+2\mid m$, the summand
\begin{equation}\label{eq:sl_2 second summand}
	 V(n-m+2)\subset \Hom(V(m),V(n))
\end{equation}
is of \emph{constant} corank one (this is precisely what our main \cref{T:main theorem} gives for the group $\SL_2$).

Using computer calculations, we have not found examples of corank at least 2 for integers $n,m\le 50$. 
This suggests that the only spaces of matrices of constant rank arising from irreducible representations of $\SL_2$ are those from \eqref{eq:sl_2 second summand} and their transposes.

\subsection{More on vector bundles}
Given a constant rank space of matrices $\varphi\colon V\hookrightarrow\Hom(V_1,V_2)$ it was explained in the introduction that it yields a morphism of vector bundles 
\[ \tilde\varphi\colon V_1\otimes\cO_{\bP(V)}\to V_2\otimes\cO_{\bP(V)}(1).\]
Denote by $\cK$ the kernel of $\tilde\varphi$, whose rank is easily computed from the dimensions of $V$, $V_1$, $V_2$ and the rank of $\tilde\varphi$.

In the case of $\SL_2$-equivariant spaces of corank one, it was shown in \cite{BFL2} that the kernel bundle is of rank $N-1$ on $\bP^N\simeq \bP(V)$ and moreover that it is indecomposable, stable and distinct from all the previously known vector bundles on $\bP^N$ of rank $N-1$.

\medskip
It is tempting to try to generalize this to other groups.
In the situation described in this paper (equivariant of corank one),
the kernel bundle is of rank bigger than $N$ if $G$ is not $\SL_2$. 
However, its dual
$\cK^\vee$ is generated by global sections, so by a classical lemma of
Serre (see \cite[Lemma 4.3.1]{OSS}) it contains a trivial sub-vector bundle whose cokernel $\overline{\cK^\vee}$ is a vector bundle of rank $N$ on $\bP^N$ that is still globally generated.
If, moreover, the $N$th Chern class of $\overline{\cK^\vee}$ is zero, then it contains a trivial line subbundle giving thus rise to a rank $N-1$ vector bundle on $\bP^N$.
The Chern classes of $\cK^\vee$ and of $\overline{\cK^\vee}$ are equal and can be computed using the Weyl character formula for irreducible representations. For corank one, the vanishing of the above $N$th Chern class happens precisely in two situations: 
\begin{itemize}
	\item representations of $SL_2$, in which case we already knew the vanishing because $\cK$ is actually of rank $N-1$, see \cite{BFL2},
	\item $G$ is arbitrary but $\mu=\nu$. Then, the corresponding space of matrices of corank one comes by restriction from the $\GL(V)$-equivariant space 
	\[ V\mapsto \Hom(V, \Lambda^2 V)\]
	where $V=V(\nu)=V(\mu)$. In this case, the above construction recovers the Tango bundle (see, for example, \cite[Theorem 4.3.3]{OSS}).
\end{itemize}
	
\bibliographystyle{amsalpha}
\bibliography{corank1spmatr.clean}	
\Addresses

\end{document}
%%% Local Variables:
%%% mode: latex
%%% TeX-master: t
%%% End: